\documentclass[10pt]{amsart}
\usepackage{graphicx, tikz}

\newcommand{\ZZ}{\mathbb{Z}}
\newcommand{\QQ}{\mathbb{Q}}

\newcommand{\HHH}{{\mathcal{H}}}

\newcommand{\W}{{\rm{Weak}}}

\newcommand{\D}{{\rm{Des}}}

\newcommand{\inv}{\operatorname{inv}}
\newcommand{\Inv}{\operatorname{Inv}}

\newcommand{\rank}{\operatorname{rank}}

\newcommand{\diameter}{\operatorname{diam}}
\newcommand{\radius}{\operatorname{rad}}
\newcommand{\Cayley}{\operatorname{Cayley}}
\newcommand{\rp}{\operatorname{right-pairs}}
\newcommand{\lp}{\operatorname{left-pairs}}

\newcommand{\Sp}{\operatorname{Supp}}

\newcommand{\area}{\operatorname{area}}
\newcommand{\bmaj}{\operatorname{bmaj}}

\newcommand{\tC}{{\tilde{C}}}
\newcommand{\e}{\epsilon}
\newcommand{\tw}{{\tilde{w}}}

\newtheorem{theorem}{Theorem}[section]
\newtheorem{corollary}[theorem]{Corollary}
\newtheorem{proposition}[theorem]{Proposition}
\newtheorem{conjecture}[theorem]{Conjecture}
\newtheorem{lemma}[theorem]{Lemma}

\newtheoremstyle{defn}{1.2ex}{1.2ex}{}{}{}{.}{.5em}%
{\textbf{\thmname{#1}\thmnumber{ #2}}\thmnote{\emph{ (#3)}}}
\theoremstyle{defn}
\newtheorem{defn}[theorem]{Definition}
\newtheorem{notation}[theorem]{Notation}
\newtheorem{example}[theorem]{Example}
\newtheorem{remark}[theorem]{Remark}
\newtheorem{question}[theorem]{Question}
\newtheorem{fact}[theorem]{Fact}
\newtheorem{observation}[theorem]{Observation}
\newtheorem{claim}[theorem]{Claim}

\numberwithin{figure}{section}
\numberwithin{table}{section}

\begin{document}
\title[Maximal chains in the non-crossing partition lattice]{On maximal chains\\in the non-crossing partition lattice}

\author{Ron M.\ Adin}
\address{Department of Mathematics\\
Bar-Ilan University\\
52900 Ramat-Gan\\
Israel} \email{radin@math.biu.ac.il}

\author{Yuval Roichman}
\address{Department of Mathematics\\
Bar-Ilan University\\
52900 Ramat-Gan\\
Israel} \email{yuvalr@math.biu.ac.il}

\keywords{Coxeter group, symmetric group, $0$-Hecke algebra,
reduced word, weak order, radius, Catalan number, Hurwitz action}


\begin{abstract}
A weak order on the set of maximal chains of the non-crossing
partition lattice is introduced and studied. A $0$-Hecke algebra
action is used to compute the radius of the graph on these chains
in which two chains are adjacent if they differ in exactly one
element.
\end{abstract}

\date{Dec.\ 23, '13}

\maketitle

\tableofcontents

\section{Introduction}
\label{section:intro}


Consider the graph $G_T(n)$ with vertex set consisting of all
maximal chains in $NC(n)$, the non-crossing partition lattice  of
type $A_{n-1}$, where two chains are adjacent if they differ in
exactly one element. This graph may be identified with the graph
of all reduced words of a given Coxeter element (long cycle) in
the symmetric group $S_n$, where the alphabet consists of all
reflections (transpositions) and two words are adjacent if they
agree in all but two adjacent letters, which are $(s, t)$ in one word
and either $(t^s, s)$ or $(t, s^t)$ in the other; here $g^h:=h^{-1}gh$. 
This graph is known to be connected -- see, e.g., \cite[Prop.\ 1.6.1]{Bessis};
we are interested in calculating its radius. This is motivated
by~\cite{AutordDehornoy} and \cite{Reiner-R}, where the analogous
question for simple reflections was studied, and by~\cite{HHMMN}
which evaluated the radius of a related graph on labeled trees.
Recall that a classical result of Hurwitz~\cite{Hurwitz} 
implies (see \cite{Denes, Strehl}) that
the number of maximal chains in 
the non-crossing partition lattice of type $A_{n-1}$ 
is equal to the number of labeled trees on $n$ vertices.


Our approach is to consider a 
$0$-Hecke algebra action on the set of maximal chains. This allows
us to define a well-behaved natural weak order on
this set. 
Each maximal interval is isomorphic to the weak order on the
symmetric group. The number of maximal elements, refined by a
generalized inversion number, is the Carlitz-Riordan $q$-Catalan
number. The resulting undirected Hasse diagram spans the graph
$G_T(n)$ of maximal chains, implying an evaluation of the radius
and an approximation of the diameter up to a factor of $3/2$.

\section{Basic concepts}
\label{section:basic_concepts}

The non-crossing partition lattice $NC(n)$, first introduced by
Kreweras~\cite{Kreweras}, may be defined as follows.

Let $T$ be the set of all reflections (transpositions) in the
symmetric group $S_n$, and let $\ell_T(\cdot)$ be the
corresponding length function: $\ell_T(\pi)$ is the minimal number
of factors in an expression of $\pi$ as a product of reflections.
Let $c$ be a Coxeter element in this group (i.e., a cycle of
length $n$); for concreteness, take $c = (1,2,\ldots,n)$. Then
$NC(n)$ is the set
$$
\{\pi\in S_n\,:\,\ell_T(\pi)+\ell_T(\pi^{-1}c)=\ell_T(c)\}
$$
ordered by
$$
\pi\le \sigma \iff
\ell_T(\pi)+\ell_T(\pi^{-1}\sigma)=\ell_T(\sigma).
$$


\begin{defn}
Let $F_n$ be the set of all maximal chains in the non-crossing partition lattice $NC(n)$.
\end{defn}

Clearly, an element of $F_n$ corresponds to a factorization of $c$
into a minimal number ($\ell_T(c) = n-1$) of transpositions. It
can thus be written in the form $(t_1,\ldots,t_{n-1})$, where $t_i
\in T$ $(1\le i\le n-1)$ and $t_1 \cdots t_{n-1} = c$.


\begin{defn}\label{d.Hurwitz_graph}
The {\em Hurwitz graph} $G_T(n)$ is the (undirected) graph with vertex set $F_n$,
where two chains are adjacent if they differ in exactly one element.
\end{defn}


The Hurwitz graph $G_T(4)$ is drawn in Figure~\ref{figure1-tik};
note that $|F_4| = 16$, and the midpoint of the figure is not a
vertex.

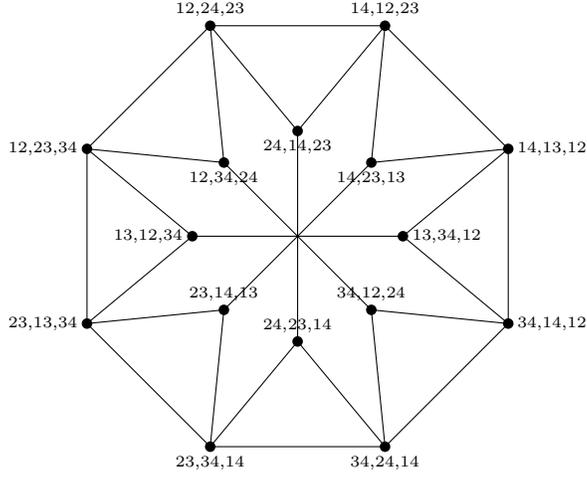
\begin{figure}[htb]
\begin{center}

\begin{tikzpicture}[scale=0.7]
\fill (-4,1.66) circle (0.1) node[left]{\tiny 12,23,34};

\fill (-1.66,4) circle (0.1) node[above]{\tiny 12,24,23};

\fill (1.66,4) circle (0.1) node[above]{\tiny 14,12,23};

\fill (4,1.66) circle (0.1) node[right]{\tiny 14,13,12};

\fill (4,-1.66) circle (0.1) node[right]{\tiny 34,14,12};

\fill (1.66,-4) circle (0.1) node[below]{\tiny 34,24,14};

\fill (-1.66,-4) circle (0.1) node[below]{\tiny 23,34,14};

\fill (-4,-1.66) circle (0.1) node[left]{\tiny 23,13,34};

\fill (-2,0) circle (0.1) node[left]{\tiny 13,12,34};

\fill (-1.4,1.4) circle (0.1) node[below]{\tiny 12,34,24};

\fill (0,2) circle (0.1) node[below]{\tiny 24,14,23};

\fill (1.4,1.4) circle (0.1) node[below]{\tiny 14,23,13};

\fill (2,0) circle (0.1) node[right]{\tiny 13,34,12};

\fill (1.4,-1.4) circle (0.1) node[above]{\tiny 34,12,24};

\fill (0,-2) circle (0.1) node[above]{\tiny 24,23,14};

\fill (-1.4,-1.4) circle (0.1) node[above]{\tiny 23,14,13};

\draw
(-4,1.66)--(-1.66,4)--(1.66,4)--(4,1.66)--(4,-1.66)--(1.66,-4)--(-1.66,-4)--(-4,-1.66)--(-4,
1.66);


\draw (-2,0)--(2,0); \draw (0,-2)--(0,2);

\draw (-1.4,-1.4)--(1.4,1.4); \draw (-1.4,1.4)--(1.4,-1.4);

\draw (-4,1.66)--(-1.4,1.4);

\draw (-1.66,4)--(-1.4,1.4);

\draw (-1.66,4)--(0,2); \draw (1.66,4)--(0,2);

\draw (4,1.66)--(1.4,1.4); \draw (1.66,4)--(1.4,1.4);

\draw (4,1.66)--(2,0); \draw (4,-1.66)--(2,0);

\draw (-4,1.66)--(-2,0); \draw (-4,-1.66)--(-2,0);

\draw (-1.66,-4)--(0,-2); \draw (1.66,-4)--(0,-2);

\draw (4,-1.66)--(1.4,-1.4); \draw (1.66,-4)--(1.4,-1.4);

\draw (-4,-1.66)--(-1.4,-1.4); \draw (-1.66,-4)--(-1.4,-1.4);


\end{tikzpicture}
\caption{\label{figure1-tik} The Hurwitz graph $G_T(4)$}
\end{center}
\end{figure}



\begin{defn}
For each $1 \le i \le n-2$ 
and $v=(t_1,\ldots,t_{n-1})\in F_n$ let
$$
R_i(v):= (t_1,\ldots,t_{i-1},t_{i+1}^{t_i},t_i,t_{i+2},\ldots,t_{n-1})
$$
and
$$
L_i(v):= (t_1,\ldots,t_{i-1},t_{i+1},t_i^{t_{i+1}},t_{i+2},\ldots,t_{n-1}),
$$
where $g^h:= h^{-1}gh$.
\end{defn}

$R_i$ and $L_i$ $(1 \le i \le n-2)$ are operators on $F_n$, where
$R_i(v)$ slides the $i$-th factor to the right
while $L_i(v)$ slides the $(i+1)$-st factor to the left.
Clearly $R_i = L_i^{-1}$ and, for every $v\in F_n$, either $R_i^2(v)=v$
(if $t_i$ and $t_{i+1}$ commute) or $R_i^3(v)=v$ (otherwise).
Furthermore, the following is true.

\begin{proposition}\label{t.Arnold}\cite{Hurwitz}
For every $1 \le i \le n-2$,
$$
R_i R_{i+1} R_i = R_{i+1} R_i R_{i+1}.
$$
\end{proposition}

This clearly defines a braid group action on $F_n$,
a distinguished example of the Hurwitz action on ordered tuples of
group elements.




\begin{observation}
For $u,v\in F_n$, $\{u,v\}$ is an edge of $G_T(n)$ if and only if there exists
an $1 \le i \le n-2$ such that either $v=R_i(u)$ or $v=L_i(u)$.
\end{observation}

\begin{remark}
The classical Tits graph of a Coxeter group is the graph whose vertices are 
all the reduced words for the longest element of the group,
in terms of the set of {\em simple} reflections $S$, and where two words are
adjacent if one is obtained from the other by a braid relation.
In other words, vertices correspond to maximal chains in the weak
order on the group, and two chains are adjacent if they differ in a ``minimal cycle''.
Replacing the set $S$ of simple reflections by the set $T$ of all reflections,
one obtains the Hurwitz graph.
The diameters of the Tits graphs of the Coxeter groups of types $A$ and $B$
were evaluated in~\cite{Reiner-R}.
Bessis proved that the Hurwitz action on $F_n$ is transitive~\cite[Prop.~1.6.1]{Bessis},
thus that the Hurwitz graph $G_T(n)$ is connected.
Its radius and
bounds on its diameter will be given in Section~\ref{section:diameter}.
\end{remark}


\begin{defn}\label{d.order}
For $u,v\in F_n$, denote $u\le v$ if $u$ belongs to a
geodesic in $G_T(n)$ from $e:=((1,2),(2,3),\ldots,(n-1,n))$ to $v$.
This will be called {\em the weak order on $F_n$}, and
the resulting poset will be denoted by $\W(F_n)$.
\end{defn}

The Hasse diagram of $\W(F_4)$  is drawn in
Figure~\ref{figure2-tik}.
This poset 
is neither a meet semilattice nor a join semilattice, as shown by the two pairs of elements
$\{(14,23, 13), (24,23,14)\}$ and $\{(23,13,34), (12,24,23)\}$.

\begin{figure}[htb]
\begin{center}

\begin{tikzpicture}[scale=0.7]
\fill (0,0) circle (0.1) node[below]{\tiny e=12,23,34};

\fill (-3,2.5) circle (0.1) node[left]{\tiny 23,13,34};

\fill (-1,2.5) circle (0.1) node[left]{\tiny 13,12,34};

\fill (1,2.5) circle (0.1) node[right]{\tiny 12,34,24};

\fill (3,2.5) circle (0.1) node[right]{\tiny 12,24,23};

\fill (-5,5) circle (0.1) node[left]{\tiny 23,14,13};

\fill (-3,5) circle (0.1) node[left]{\tiny 23,34,14};

\fill (-1,5) circle (0.1) node[left]{\tiny 13,34,12};

\fill (1,5) circle (0.1) node[right]{\tiny 34,12,24};

\fill (3,5) circle (0.1) node[right]{\tiny 14,12,23};

\fill (5,5) circle (0.1) node[right]{\tiny 24,14,23};

\fill (-4,7.5) circle (0.1) node[above]{\tiny 14,23,13};

\fill (-2,7.5) circle (0.1) node[above]{\tiny 14,13,12};

\fill (0,7.5) circle (0.1) node[above]{\tiny 34,14,12};

\fill (2,7.5) circle (0.1) node[above]{\tiny 34,24,14};

\fill (4,7.5) circle (0.1) node[above]{\tiny 24,23,14};

\draw
(0,0)--(-3,2.5)--(-5,5)--(-4,7.5)--(3,5);

\draw (0,0)--(3,2.5)--(5,5)--(4,7.5)--(-3,5);

\draw (0,0)--(-1,2.5)--(-1,5)--(0,7.5)--(1,5)--(1,2.5)--(0,0);

\draw (-3,2.5)--(-3,5);

\draw (3,2.5)--(3,5);

\draw (-1,5)--(-2,7.5)--(3,5);

\draw (1,5)--(2,7.5)--(-3,5);


\end{tikzpicture}
\caption{\label{figure2-tik} The Hasse diagram of $\W(F_4)$}
\end{center}
\end{figure}
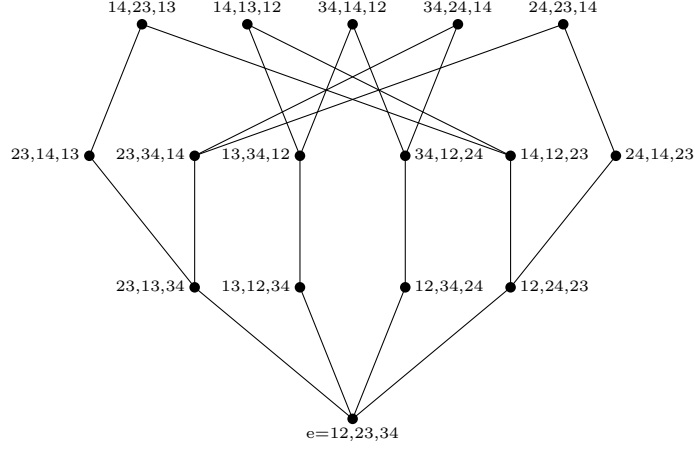



\begin{remark}\label{weak-remark}
To motivate the name ``weak order'' consider
%
a Coxeter group $W$, with standard generating set $S$, 
acting transitively on a set $X$. Choose an element $e$ in $X$.
For every $y\in X$, let $\ell_S(y)$ be the distance from
$e$ to $y$ in the associated Schreier graph $\Gamma(W/Stab(e), S)$;
and for every $y,z\in X$, define $z$ to {\em cover} $y$ if there
exists some $s\in S$ such that $z=s(y)$ and $\ell_S(z)=\ell_S(y)+1$.
This defines a natural ``weak order'' on $X$, with $e$ as a minimal element;
%
analogous concepts of strong and absolute orders on $W$-sets were
studied in~\cite{Rains-Vazirani, Athanasiadis-R}.
%
%
%
%
%
%
In our case, by Proposition~\ref{t.Arnold}
$F_n$ carries a transitive action of the braid group of type $A$
which covers the Coxeter group $S_{n-1}$,  
and can thus be equipped with a similar order. 
%
Furthermore, it will be shown later that maximal intervals in this
order are isomorphic to the weak order on $S_{n-1}$.
%
\end{remark}


The poset $\W(F_n)$ has several nice properties, such as a $0$-Hecke algebra action,
the above mentioned structure of maximal intervals, and a Catalan number of maximal elements.
It will be used to evaluate the radius and bound the diameter of the graph $G_T(n)$.

\section{A map from $F_n$ to $S_{n-1}$}
\label{section:map}

Let $F_n$ be as in Section~\ref{section:intro}.
Define a map $\phi: F_n\rightarrow S_{n-1}$ as follows.
\begin{defn}\label{d.sigma}
For $w=(t_1,\ldots,t_{n-1})\in F_n$ define the partial products
$$
\sigma_j:= t_j t_{j+1}\cdots t_{n-1} \qquad (1\le j\le n-1)
$$
and the empty product
$$
\sigma_n:= id.
$$
\end{defn}
By definition, $\sigma_j= t_j \sigma_{j+1}$ for every $1\le j\le n-1$.
Thus $\sigma_j(i)\ne \sigma_{j+1}(i)$ for exactly two values of $1\le i\le n$,
and $\sigma_j(i)> \sigma_{j+1}(i)$ for exactly one such value.
Ignoring $i=n$, the set
$$
A_j:=\{1\le i\le n-1\,|\,\sigma_j(i)>\sigma_{j+1}(i)\}
$$
thus satisfies
\begin{equation}\label{e.1}
\left|A_j\right|\le 1\qquad(1\le j\le n-1).
\end{equation}
On the other hand, $\sigma_1(i) = c(i) = i+1$ while $\sigma_n(i)=i$,
so that $i\in \bigcup_{j=1}^{n-1} A_j$ for every $1\le i \le n-1$.
It follows that
\begin{equation}\label{e.2}
\left|\bigcup_{j=1}^{n-1} A_j\right|\ge n-1.
\end{equation}
Combining (\ref{e.1}) with (\ref{e.2}) one concludes that
$$
\left|\bigcup_{j=1}^{n-1} A_j\right|= n-1
$$
and
$$
\left|A_j\right|=1 \qquad(\forall j).
$$
It follows that the map $\pi_w$ defined by
$$
\pi_w (j):= i \text{\ if\ } A_j=\{i\}
$$
is a permutation in $S_{n-1}$.
\begin{defn}
Define $\phi:F_n \to S_{n-1}$ by
$$
\phi(w):=\pi_w\qquad(\forall w\in F_n).
$$
\end{defn}

\begin{observation}\label{obs.phi}
For every $w\in F_n$ and $1\le j\le n-1$, if $t_j=(a,b)$ with $a<b$ then
$$
\phi(w)(j)= \sigma_{j+1}^{-1}(a).
$$
\end{observation}

\medskip

\begin{example}
Let $w=((2,5),(1,5),(2,4),(2,3))\in F_5$. Then $\sigma_5=id$, $\sigma_4=(2,3)$,
$\sigma_3=(2,3,4)$, $\sigma_2=(2,3,4)(1,5)$ and $\sigma_1=(1,2,3,4,5)=c$,
where permutations are written in cycle notation. Thus
$\phi(w)(1) = 4 = \sigma_2^{-1}(2)$,
$\phi(w)(2) = 1 = \sigma_3^{-1}(1)$,
$\phi(w)(3) = 3 = \sigma_4^{-1}(2)$
and $\phi(w)(4) = 2 = \sigma_5^{-1}(2)$.
\end{example}

\bigskip

For a sequence $w = ((a_1,b_1), \ldots, (a_{\ell},b_{\ell}))$ of $\ell$
transpositions in $S_n$, let $G(w)$ be the geometric graph
whose vertices are $\{1, \ldots, n\}$ drawn on a circle and
whose edges are $\{\{a_j,b_j\}:\ 1 \le j \le \ell\}$. The following characterization
of the words in $F_n$, due to Goulden and Yong, is very useful.

\begin{proposition}\label{GY22} \cite[Theorem 2.2]{Goulden-Yong}\\
For $t_1, \ldots, t_{n-1} \in T(S_n)$,
the word $w = (t_1, \ldots, t_{n-1})$ belongs to $F_n$ if and only if
the following three conditions hold:
\begin{itemize}
\item[(i)]
$G(w)$ is a tree.
\item[(ii)]
$G(w)$ is non-crossing; namely, two edges may intersect only in a common vertex.
\item[(iii)] Cyclically decreasing neighbors:
For every $1 \le a \le n$ and $1 \le i < j \le n-1$, if $t_i = (a,c)$ and $t_j = (a,b)$ then $c >_a b$.
Here $<_a$ is the linear order $a <_a a+1 <_a \cdots <_a n <_a 1 <_a \cdots <_a a-1$.
\end{itemize}
\end{proposition}


Proposition~\ref{GY22}(iii) may be stated in the following equivalent form.
\begin{corollary}\label{t.avoid}
For $w = (t_1, \ldots, t_{n-1}) \in F_n$ and $i < j$,
if $t_i$ and $t_j$ do not commute then the pair $(t_i, t_j)$ is 
either $((a,c),(a,b))$, $((b,c),(a,c))$ or $((a,b),(b,c))$,
for some $a < b < c$.
\end{corollary}

\medskip


\begin{lemma}\label{phi.local-range}
For $w = (t_1, \ldots, t_{n-1}) \in F_n$ and $1 \le j \le n-1$,
if $t_j = (a,b)$ with $a < b$ then 
$i := \phi(w)(j)$ satisfies: $a \le i < b$. 
\end{lemma}
\begin{proof}
The discussion leading to the definition of $A_j$ and $\phi$ above shows that,
for any $1 \le i \le n-1$, the sequence $i = \sigma_n(i), \sigma_{n-1}(i), \ldots, \sigma_1(i) = i+1$
is weakly decreasing except for a single step $j$ for which,
if $a = \sigma_{j+1}(i) < \sigma_{j}(i) = b$, it follows that $t_j = (a,b)$ and $\phi(w)(j) = i$.
In particular, it follows that $i \ge a$ and $b \ge i+1$.

\end{proof}

\begin{lemma}\label{phi.local-range-2}
For $w = (t_1, \ldots, t_{n-1}) \in F_n$ and $j \ne k$,
if $t_j = (a,d)$ and $t_k = (b,c)$ with $a \le b < c \le d$ then 
$i := \phi(w)(j)$ satisfies:
\begin{itemize}
\item[(1)]
Either $a \le i < b$ or $c \le i < d$. 
\item[(2)]
Deleting the edge $\{a,d\}$ from the tree $G(w)$ leaves two connected
components, $T_a$ (containing $a$) and $T_d$ (containing $d$); 
only one of them contains the edge $\{b, c\}$. 
If this is $T_d$ then $a \le i < b$, otherwise $c \le i < d$.
\end{itemize}
\end{lemma}
\begin{proof}
By Lemma~\ref{phi.local-range}, $a \le i < d$.
By the proof of that Lemma, the directed path in $G(w)$ leading from $i$ to $i+1$
consists of a single increasing step (edge) $a \to d$, together with some decreasing edges. 
Assuming $b \le i < c$, this path cannot contain the step $c \to b$, and also not $b \to c$
(since $j \ne k$ implies $(a,d) \ne (b,c)$). However, since $G(w)$ is non-crossing
and $b \le i < i+1 \le c$, this path must pass through $b$ and $c$.
We thus get two distinct paths in $G(w)$ from $b$ to $c$, contradicting Proposition~\ref{GY22}(i).
Thus either $a \le i < b$ or $c \le i < d$.

The rest of the claim now follows from the observation that $i$ and $i+1$ belong to
distinct connected components of $G(w)$ minus the edge $\{a, d\}$. 
If $T_d$ contains the edge $\{b, c\}$ then, by the non-crossing property, 
it contains all the vertices from $b$ up to $d$, and therefore $a \le i < b$.
The complementary case leads, similarly, to the complementary conclusion.

\end{proof}

\begin{lemma}\label{t.inversions} {\rm (Characterization of inversions in $\phi(w) \in S_{n-1}$)}\\
For $w = (t_1, \ldots, t_{n-1}) \in F_n$ and $j < k$,
$\phi(w)(j) > \phi(w)(k)$ if and only if one of the following holds:
\begin{itemize}
\item[(1)]
$\exists a<b<c$ such that either:
   \begin{itemize}
   \item[(i)]
   $t_j = (a,c)$ and $t_k = (a,b)$; or:
   \item[(ii)]
   $t_j = (b,c)$ and $t_k = (a,c)$.
   \end{itemize}
\item[(2)] 
$\exists a<b<c<d$ such that $t_j = (c,d)$ and $t_k = (a,b)$. 
\item[(3)] 
$\exists a<b<c<d$ such that either:
   \begin{itemize}
   \item[(i)]
   $t_j = (a,d)$, $t_k = (b,c)$ and, 
   in the notation of Lemma~\ref{phi.local-range-2}, 
   the edge $\{b, c\}$ belongs to $T_a$; or
   \item[(ii)]
   $t_j = (b,c)$, $t_k = (a,d)$ and,
   in the notation of Lemma~\ref{phi.local-range-2}, 
   the edge $\{b, c\}$ belongs to $T_d$.
\end{itemize}
\end{itemize}
\end{lemma}
\begin{proof} 
Assume that $j < k$ and $\phi(w)(j) > \phi(w)(k)$.
If $t_j$ and $t_k$ do not commute then, by Corollary~\ref{t.avoid},
there are 3 options for the pair $(t_j, t_k)$.
Two of them appear in $(1)$ above while the third, $(t_j, t_k) = ((a,b), (b,c))$,
is ruled out since then, by Lemma~\ref{phi.local-range}, $\phi(w)(j) < b \le \phi(w)(k)$.
If $t_j$ and $t_k$ commute then there are 4 options for the pair $(t_j, t_k)$:
$((a,b), (c,d))$, $((c,d), (a,b))$, $((a,d), (b,c))$ and $((b,c), (a,d))$ (for $a < b < c < d$).
The first option is ruled out by Lemma~\ref{phi.local-range}, 
the second appears in $(2)$ above,
and the third and fourth appear in $(3)$ above in the restricted versions allowed by
Lemmas~\ref{phi.local-range} and~\ref{phi.local-range-2}.
 
Conversely, if $j < k$ and one of the options $(1)$, $(2)$ and $(3)$ holds, 
then Lemmas~\ref{phi.local-range} and~\ref{phi.local-range-2}
show that indeed $\phi(w)(j) > \phi(w)(k)$.

\end{proof}


Recall the operators $R_j$ and $L_j$ from Section~\ref{section:basic_concepts}.
By definition~\ref{d.sigma}, 
for all $1 \le j \le n-2$ and $1 \le k \le n$,
\[
\sigma_k(R_j(w)) =
\begin{cases}
t_j t_{j+1} \sigma_k(w), &\text{\rm if } k = j+1; \\
\sigma_k(w), &\text{\rm otherwise}
\end{cases}
\]
and
\[
\sigma_k(L_j(w)) =
\begin{cases}
t_{j+1} t_j \sigma_k(w), &\text{\rm if } k = j+1; \\
\sigma_k(w), &\text{\rm otherwise}.
\end{cases}
\]

\begin{lemma}\label{t.main-lemma}
For every $w\in F_n$ and $1\le j \le n-2$,
\[
\phi(R_j(w)) =
\begin{cases}
\phi(w), &\text{\rm if\ } \exists\, a<b<c \text{\rm\ \ s.t.\ } t_j=(a,c) \text{\rm\ and\ } t_{j+1}=(a,b) ;\\
\phi(w) s_j, &\text{\rm otherwise}
\end{cases}
\]
and
\[
\phi (L_j(w))=
\begin{cases}
\phi(w), &\text{\rm if\ } \exists\, a<b<c \text{\rm\ \ s.t.\ } t_j=(b,c) \text{\rm\ and\ }  t_{j+1}= (a,c) ;\\
\phi(w) s_j, &\text{\rm otherwise.}
\end{cases}
\]
\end{lemma}


\begin{proof}
We prove the formula for $R_j$; 
the proof of the other formula 
is similar.

Since $t_k(R_j(w)) = t_k(w)$ for $k \ne j, j+1$ 
and $\sigma_k(R_j(w)) = \sigma_k(w)$ for $k \ne j+1$
it follows, by Observation~\ref{obs.phi}, that $\phi(R_j(w))(k)= \phi(w)(k)$
for all $k \ne j, j+1$.
Thus either $\phi(R_j(w))=\phi(w)$ or $\phi(R_j(w))=\phi(w) s_j$.
By Observation~\ref{obs.phi} (for $j+1$), $\phi(R_j(w))=\phi(w)$ if and only if
$t_{j+1}=(a,b)$ and $t_j=(a,c)$ for some $a$, $b$ and $c$ such that $a<b$ and $a<c$.
These $t_j$ and $t_{j+1}$ do not commute, and by Corollary~\ref{t.avoid} necessarily $b<c$ as well.

\end{proof}

\section{Properties of the weak order on $F_n$}
\label{section:order}

Recall the graded poset $\W(F_n)$ from
Section~\ref{section:basic_concepts}.
Note that, by definition,

\begin{fact}\label{t.minimum}
The word $e=((1,2),(2,3),\ldots, (n-1,n))$ is the unique minimal element in
$\W(F_n)$.
\end{fact}


\begin{lemma}\label{t.unique_id}
$$
\phi^{-1}(id) = \{e\}.
$$
\end{lemma}
\begin{proof}
Clearly $\phi(e) = id$, since for $e = ((1,2),\ldots,(n-1,n))$
one has $t_j = (j,j+1)$, $\sigma_j = t_j \cdots t_{n-1} = (j,\ldots,n)$
and thus $\phi(e)(j) = \sigma_{j+1}^{-1}(j) = j$ by Observation~\ref{obs.phi}.

Conversely, assume that $w = (t_1,\ldots,t_{n-1})\in F_n$ has $\phi(w) = id$.
Then $\phi(w)(n-1) = n-1$ implies that $\sigma_{n-1}(n-1) > \sigma_n(n-1) = n-1$,
so that $\sigma_{n-1}(n-1) = n$ and $t_{n-1} = \sigma_{n-1} = (n-1,n)$.
It follows that $t_1 \cdots t_{n-2} = (1,\ldots,n)(n-1,n) = (1,\ldots,n-1)$,
so that necessarily $t_1,\ldots,t_{n-2}\in S_{n-1}$ and
$w' := (t_1,\ldots,t_{n-2})$ belongs to $F_{n-1}$, with
$\phi(w') = id\in S_{n-2}$. Induction on $n$ completes the proof.

\end{proof}

For any $w \in F_n$, let $\rank(w)$ be its rank in the poset $\W(F_n)$.
For $\pi \in S_{n-1}$, let $\inv(\pi)$ be its inversion number.
\begin{lemma}\label{t.rank_eq_inv}
For any $w\in F_n$,
$$
\rank(w) = \inv(\phi(w)).
$$
\end{lemma}
\begin{proof}
By induction on $\inv(\phi(w))$.

If $\inv(\phi(w)) = 0$ then $\phi(w) = id$. By Lemma~\ref{t.unique_id} $w = e$,
so that indeed $\rank(w) = 0 = \inv(\phi(w))$.

For the induction step let $\inv(\phi(w)) = \ell > 0$, and assume that 
the claim holds for $u\in F_n$ whenever $\inv(\phi(u)) < \ell$. 
Let $\pi := \phi(w)\in S_{n-1}$. Since $\inv(\pi) > 0$
there exists $1\le j\le n-2$ such that $\inv(\pi s_j) < \inv(\pi)$. 
By Lemma~\ref{t.main-lemma}, $\phi(R_j(w)) = \pi s_j$
unless $\exists\, a<b<c$ such that $t_j=(a,c)$ and $t_{j+1}=(a,b)$ (in $w$); 
whereas $\phi(L_j(w)) = \pi s_j$ unless $\exists\, a<b<c$ such that 
$t_j=(b,c)$ and $t_{j+1}=(a,c)$. These two exceptions cannot hold simultaneously, 
and thus $\phi(u) = \pi s_j$ for either $u = R_j(w)$ or $u = L_j(w)$. 
For this $u$, $\inv(\phi(u)) = \inv(\pi s_j) = \ell-1$, so by the induction
hypothesis $\rank(u) = \inv(\phi(u))$. 
Since $u$ and $w$ are connected by an arc in the graph $G_T(n)$, it follows that
$\rank(w) \le \rank(u) + 1 = \ell$. 
On the other hand, in a geodesic from $e$ to $w$ in $G_T(n)$ the value of
$\inv(\phi(\cdot))$ changes by $0$ or $1$ on each arc (again by Lemma~\ref{t.main-lemma}), 
so that $\rank(w) \ge \inv(\phi(w)) = \ell$. Thus $\rank(w) = \ell = \inv(\phi(w))$.

\end{proof}

\begin{corollary}\label{t.Hasse}
The (undirected) Hasse diagram of $\W(F_n)$ is obtained from the graph $G_T(n)$
by deleting all edges connecting a vertex $(\ldots,(a,c),(a,b),\ldots)$ with a vertex
$(\ldots,(b,c),(a,c),\ldots)$ for some $a<b<c$.
\end{corollary}
\begin{proof}
By Lemma~\ref{t.main-lemma} and Lemma~\ref{t.rank_eq_inv}.
\end{proof}


%
%

\begin{lemma}\label{t.forbidden-wedge}
For every $w \in F_n$ and $1 \le j \le n-2$,
if $R_j(w) < w$ and $L_j(w) < w$ in $\W(F_n)$ then $R_j(w) = L_j(w)$.
\end{lemma}

\begin{proof}
We can assume that $t_j$ and $t_{j+1}$ do not commute,
since otherwise $R_j(w) = L_j(w)$.
Then, by Corollary~\ref{t.avoid},
\[
(t_j, t_{j+1})\in \{((a,c),(a,b)),\,((b,c),(a,c)),\,((a,b),(b,c))\}
\]
for some $a < b < c$.
The same holds for the $j$-th and $(j+1)$-st factors of $R_j(w)$,
as well as for those of $L_j(w)$.


By Lemma~\ref{t.rank_eq_inv}, $R_j(w) < w$ implies that
$\inv(\phi(R_j(w))) < \inv(\phi(w))$, which is equivalent, by
Lemma~\ref{t.main-lemma}, to $\phi(R_j(w))(j) < \phi(R_j(w))(j+1)$.
By Lemma~\ref{t.inversions}
this rules out $((a,c), (a,b))$ and $((b,c), (a,c))$,
and leaves only $((a,b), (b,c))$ as an option for 
the $j$-th and $(j+1)$-st factors of $R_j(w)$.
A similar conclusion holds for $L_j(w)$, and this implies that
$(t_j, t_{j+1}) = ((b,c), (a,c))$ and simultaneously
$(t_j, t_{j+1}) = ((a,c), (a,b))$, a contradiction.



\end{proof}


\begin{remark}\label{t.RL-cover}
Lemma~\ref{t.forbidden-wedge} may be stated in the following two equivalent forms.
\begin{itemize}
\item
For every $w \in F_n$ and $1 \le j \le n-2$, $R_j(w) > w$ if and only if $L_j(w) > w$.
\item
Each orbit in $F_n$ of each $R_j$ (or $L_j$) consists of
either a pair $\{w_1, w_2\}$ with $w_1 < w_2$
or a triple $\{w_1, w_2, w_3\}$ with $w_1 < w_2$ and $w_1 < w_3$.
\end{itemize}
\end{remark}


\medskip

\section{$0$-Hecke algebra action}
\label{section:Hecke}

The {\em $0$-Hecke algebra} $\HHH_n(0)$ is an algebra over
$\QQ$ generated by $\{T_i:\ 1 \le i \le n-1\}$ with defining
relations
\begin{equation}\label{e.0_Hecke_relation}
T_i^2=T_i \qquad (1\le i \le n-1),
\end{equation}
\begin{equation}
T_i T_j=T_j T_i \qquad (|i-j|>1)
\end{equation}
and
\begin{equation}
T_i T_{i+1} T_i = T_{i+1} T_i T_{i+1} \qquad (1 \le i \le n-2).
\end{equation}

$\HHH_n(0)$ can be viewed as the specialization $q=0$ of the
Iwahori-Hecke algebra $\HHH_n(q)$, deforming the group algebra of
$S_n$, or as the semigroup algebra of a suitable monoid.
Note that our notation is slightly non-standard, since equation~(\ref{e.0_Hecke_relation})
is usually stated as $T_i^2 = - T_i$, which amounts to replacing $T_i$ by $-T_i$ (for all $i$).


A faithful action of the $0$-Hecke algebra $\HHH_{n-1}(0)$ on $F_n$ is introduced in this section.
This action will be applied to show that every maximal interval in $\W(F_n)$ is isomorphic to
the weak order on the symmetric group $S_{n-1}$.

\medskip

\subsection{Down operators}\label{section:down}
Recall the descent set of a permutation $\pi\in S_{n-1}$,
$\D(\pi):=\{i:\ \pi (i)>\pi(i+1)\}.$

\begin{defn}\label{def-down}\ \rm
For $1\le i\le n-2$ define the {\em $i$-th down operator}\  $D_i: F_n \longrightarrow F_n$ by
\[
D_i(w):=\begin{cases}
   R_i (w), & \text{ if } i\in \D(\phi(w)) \text{ and } R_i (w)<w;\\
   L_i (w), & \text{ if } i\in \D(\phi(w)) \text{ and } L_i (w)<w;\\
   w,  & \text{ if } i\not\in \D(\phi(w)).
\end{cases}
\]
\end{defn}


\begin{claim} The down operators are well defined.
\end{claim}

\begin{proof}
By Lemma~\ref{t.main-lemma}, for any $w \in F_n$ and $1 \le i \le n-2$,
either $\phi(R_i(w)) = \phi(w) s_i$ or $\phi(L_i(w)) = \phi(w) s_i$ (or both).
If $i\in \D(\phi(w))$ then $\inv(\phi(w)s_i)<\inv(\phi(w))$.
Thus, by Lemma~\ref{t.rank_eq_inv},
either $\rank(R_i(w)) < \rank(w)$ or $\rank(L_i(w)) < \rank(w)$,
namely: Either $R_i(w) < w$ or $L_i(w) < w$ in $\W(F_n)$.
 If both hold then, by Lemma~\ref{t.forbidden-wedge}, $R_i(w)=L_i(w)$.

\end{proof}


\begin{observation}\label{obs.down}
For every $1\le i \le n-2$ and $w\in F_n$,
\[
\phi(D_i(w))=\begin{cases}
   \phi(w) s_i, & \text{ if } i\in \D(\phi(w));\\
   \phi(w),  & \text{ if } i\not\in \D(\phi(w)).
\end{cases}
\]
\end{observation}


\begin{lemma}\label{t.down-relations}
The down operators $\{D_i\,:\,1 \le i \le n-2\}$ satisfy the defining
relations of the $0$-Hecke algebra $\HHH_{n-1}(0)$.
\end{lemma}

\begin{proof}
If $i\not\in \D(\phi(w))$ then, by Definition~\ref{def-down}, $D_i(w)=w$.
If $i\in \D(\phi(w))$ then, by Observation~\ref{obs.down}, $i\not\in \D(\phi(D_i(w)))$
and therefore, by Definition~\ref{def-down}, $D_i(D_i(w)) = D_i(w)$. Thus
$D_i^2=D_i$ for every $1 \le i \le n-2$.

\medskip

If $|i-j| > 1$ then, by Definition~\ref{def-down}, $D_i D_j = D_j D_i$.

\medskip

We have to verify that $D_i D_{i+1} D_i= D_{i+1} D_i D_{i+1}$ for every $1 \le i \le n-3$.
%
%
If $i, i+1 \not\in \D(\phi(w))$ then, by Definition~\ref{def-down},
$D_i D_{i+1} D_i(w) = w = D_{i+1} D_i D_{i+1}(w)$.
%
%
%
%

If $i\in \D(\phi(w))$ but $i+1\not\in \D(\phi(w))$ then, by Definition~\ref{def-down}, 
$D_{i+1}(w) = w$ and therefore $D_{i+1}D_i D_{i+1}(w) = D_{i+1} D_i(w)$.
By Observation~\ref{obs.down}, the permutation $\phi(D_i(w))$ is the same as $\phi(w)$
except that $\phi(D_i(w))(i) < \phi(D_i(w))(i+1)$, which is not necessarily the case for $\phi(w)$. 
The assumption $i\in \D(\phi(w))$ but $i+1\not\in \D(\phi(w))$ implies that
$\phi(w)(i+1) = \min \{ \phi(w)(i), \phi(w)(i+1), \phi(w)(i+2)\}$.
Denoting $\tw := D_{i+1} D_i(w)$ we get $\phi(\tw)(i) = \phi(i+1)$ and therefore
$\phi(\tw)(i) = \min \{ \phi(\tw)(i), \phi(\tw)(i+1), \phi(\tw)(i+2)\}$.
Hence $i \not\in \D(\phi(D_{i+1} D_i(w)))$ and, by definition~\ref{def-down},
$D_i D_{i+1} D_i(w) = D_{i+1} D_i(w)$ as well.
%
%

For similar reasons, if $i\not\in \D(\phi(w))$ and $i+1\in \D(\phi(w))$ then
$D_{i}(w) = w$ and $i+1 \not\in \D(\phi(D_i D_{i+1}(w)))$, so that
$D_i D_{i+1}D_i(w) = D_{i} D_{i+1}(w) = D_{i+1}D_i D_{i+1}(w)$.

\medskip

It remains to verify the braid relation $D_i D_{i+1} D_i= D_{i+1} D_i D_{i+1}$
when $i, i+1 \in \D(\phi(w))$.
By assumption $\phi(w)(i) > \phi(w)(i+1) > \phi(w)(i+2)$, so that
$w$, $D_i(w)$, $D_{i+1} D_i(w)$ and $D_i D_{i+1} D_i(w)$ are all distinct:
$w > D_i(w) > D_{i+1} D_i(w) > D_i D_{i+1} D_i(w)$ in $\W(F_n)$.
Similarly with $D_i$ and $D_{i+1}$ interchanged.
For a word $w=(t_1,\dots,t_{n-1})\in F_n$ let $w_i:=t_i=(a_i,b_i)$.
Obviously, if $j \not\in \{i,i+1,i+2\}$ then 
$D_i D_{i+1} D_i (w)_j = D_{i+1} D_i D_{i+1}(w)_j = t_j$. 
It thus suffices to verify that 
$D_i D_{i+1} D_i (w)_j = D_{i+1} D_i D_{i+1}(w)_j$ for $j \in \{i,i+1,i+2\}$.
For a subset $I \subseteq \{1, \ldots, n-1\}$ let $\Sp_I(w) := \cup_{i \in I}\{a_i, b_i\}$.
By Proposition~\ref{GY22}(i), $4 \le |\Sp_{\{i, i+1, i+2\}}(w)| \le 6$.
We shall consider these $3$ cases separately.

\medskip

\noindent{\bf Case 1.}  $|\Sp_{\{i,i+1,i+2\}}(w)|=6$. \\
In this case $t_i, t_{i+1}$ and $t_{i+2}$ commute, so that 
$D_i(w) = R_i(w) = L_i(w)$ etc. Thus, by Proposition~\ref{t.Arnold},
\[
D_i D_{i+1} D_i (w) = 
R_i R_{i+1} R_i (w) = 
R_{i+1} R_i R_{i+1}(w) =
D_{i+1} D_i D_{i+1}(w).
\]

\medskip

\noindent{\bf Case 2.} $|\Sp_{\{i,i+1,i+2\}}(w)|=5$. \\
In this case one of the transpositions $t_i, t_{i+1}$ or $t_{i+2}$ 
commutes with other two.
If $t_i$ commutes with $t_{i+1}$ and $t_{i+2}$ then,
since $i+1\in \D(\phi(w))$, if $\Sp_{\{i+1,i+2\}(w)} = \{a,b,c\}$
with $a<b<c$ then, by Lemma~\ref{t.inversions}(1), 
either $(t_{i+1}, t_{i+2}) = ((a,c), (a,b))$ 
or $(t_{i+1}, t_{i+2}) = ((b,c), (a,c))$
and in both cases, after application of $D_{i+1}$,
this becomes $((a,b), (b,c))$.
Thus
\begin{eqnarray*}
D_{i+1} D_i D_{i+1}(w)
&=& D_{i+1} D_i D_{i+1}(\ldots, t_{i}, t_{i+1}, t_{i+2}, \ldots) \\
&=& D_{i+1} D_i (\ldots,  t_i, (a,b), (b,c), \ldots) \\
&=& (\ldots, (a,b), (b,c), t_i, \ldots)
\end{eqnarray*}
and also
\begin{eqnarray*}
D_i D_{i+1} D_i(w)
&=& D_i D_{i+1} D_i(\ldots, t_{i}, t_{i+1}, t_{i+2}, \ldots) \\
&=& D_i (\ldots, t_{i+1}, t_{i+2}, t_i, \ldots) \\
&=& (\ldots, (a,b), (b,c), t_i, \ldots).
\end{eqnarray*}
Similarly, if $t_{i+2}$ commutes with $t_i$ and $t_{i+1}$ then,
if $\Sp_{\{i, i+1\}}(w) = \{a, b, c\}$ with $a<b<c$,
\[
D_i D_{i+1} D_i(w) 
= D_{i+1} D_i D_{i+1}(w)
= (\ldots, t_{i+2}, (a,b), (b,c), \ldots).
\]
If $t_{i+1}$ commutes with $t_i$ and $t_{i+2}$ then,
if $\Sp_{\{i, i+2\}}(w) = \{a,b,c\}$ with $a<b<c$,
\[
D_i D_{i+1} D_i(w)
= D_{i+1} D_i D_{i+1}(w)
= (\ldots, (a,b), t_{i+1}, (b,c), \ldots).
\]

\medskip

\noindent{\bf Case 3.} $|\Sp_{\{i,i+1,i+2\}}(w)|=4$. \\
Let $\Sp_{\{i,i+1,i+2\}}(w) = \{j_1, j_2, j_3, j_4\}$ with $j_1 < j_2 < j_3 < j_4$. 
By Proposition~\ref{GY22} and Lemma~\ref{t.inversions}(1) 
there are exactly $5$ possible options for $(t_i, t_{i+1}, t_{i+2})$, 
which correspond to the $5$ maximal elements of $F_4$, as in Figure~\ref{figure2-tik}. 
Inspecting this figure, it is easy to verify that the braid relation hold in all cases.


\end{proof}


\begin{defn}\ \rm
For a permutation $\pi \in S_{n-1}$ let $s_{i_1} \cdots s_{i_k}$
be a reduced word for $\pi$ in the alphabet of simple reflections.
Define the operator $D_\pi:F_n \longrightarrow F_n$ by
$$
D_\pi:=D_{i_1}\cdots D_{i_k}.
$$
\end{defn}

\begin{claim}
The operator $D_\pi$ is well-defined.
\end{claim}

\begin{proof}
Since any two reduced words for the same permutation are connected
by a sequence of braid relations, the result follows from
Lemma~\ref{t.down-relations}.
\end{proof}




\medskip

\subsection{Maximal intervals in the weak order on $F_n$}
\label{section:intervals}

\begin{lemma}\label{interval1}
For every element $w \in F_n$, the lower interval $[e,w]$ is exactly
$\{D_\pi(w)\,:\,\pi \in S_{n-1}\}$.
\end{lemma}

\begin{proof}
By definition, $D_i(w) \le w$ for any $1 \le i \le n-2$ and therefore
$D_\pi(w) \le w$ for any $\pi \in S_{n-1}$.

In the other direction, if $v \le w$ then
$v$ belongs to a geodesic from $w$ to $e$ in $G_T(n)$, part of which is a saturated chain
$w = w_0 > w_1 > \ldots > w_k = v$ in $\W(F_n)$.
Clearly each step is $w_j = D_{i_j}(w_{j-1})$ for some $i_j$ ($1\le j \le k$).
If the word $s_{i_k} \cdots s_{i_1}$ is not reduced then it can be brought, by relations
as in Lemma~\ref{t.down-relations}, to a reduced word for some $\pi$, so that
$v = D_\pi(w)$.

\end{proof}


%
\begin{lemma}\label{interval2}
Let $w_0$ be any maximal element in $\W(F_n)$. Then, for any $\pi\in S_{n-1}$,
\[
\phi(D_{\pi}(w_0))=\pi_0 \pi^{-1},
\]
where $\pi_0 = [n-1, \ldots, 1]$ is the maximal element in the weak order on $S_{n-1}$.
\end{lemma}

\begin{proof}
By induction on $\inv(\pi)$. If $\pi=id$ then
necessarily $\phi(D_{\pi}(w_0))=\phi(w_0)=\pi_0$,
since otherwise there exists $1 \le j \le n-2$ such that $\inv(\phi(w_0)s_j)>\inv(\phi(w_0))$;
and then, by Lemma~\ref{t.main-lemma} and Lemma~\ref{t.rank_eq_inv}, 
either $v = R_j(w_0)$ or $v = L_j(w_0)$ has $\rank(v) > \rank(w_0)$,
contradicting the maximality of $w_0$.


If $\pi \ne id$ then there exists $1 \le j \le n-2$ such that 
$\sigma := s_j \pi$ has $\inv(\sigma)=\inv(\pi)-1$. 
Then $\pi_0 \pi^{-1}=\pi_0 \sigma^{-1} s_j$  has 
$\inv(\pi_0 \pi^{-1}) = \inv(\pi_0 \sigma^{-1}) - 1$ with $j\in \D(\pi_0\sigma^{-1})$. 
By the induction hypothesis $\phi(D_{\sigma}(w_0))=\pi_0 \sigma^{-1}$, 
thus $j\in \D(\phi(D_{\sigma}(w_0)))$. Combining this with
Observation~\ref{obs.down} we obtain
\[
\phi(D_{\pi}(w_0)) = \phi(D_j D_{\sigma}(w_0)) =
\phi(D_{\sigma}(w_0)) s_j = \pi_0 \sigma^{-1}s_j = \pi_0 \pi^{-1}.
\]

\end{proof}


We conclude

\begin{theorem}\label{t.main2}
Each maximal interval in $\W(F_n)$ is isomorphic, as a poset, to
the weak order on the symmetric group $S_{n-1}$.
\end{theorem}

\begin{proof}
By Lemmas~\ref{interval1} and~\ref{interval2},
$\phi$ is onto $S_{n-1}$ even when restricted to a maximal interval $[e, w_0]$.


To prove that $\phi$  is also one-to-one when restricted to $[e, w_0]$,
it suffices to show that the cardinality of the interval $[e, w_0]$ 
does not exceed the cardinality of $S_{n-1}$;
and this follows again from Lemma~\ref{interval1}. 


Finally, to prove that $\phi$ is a poset isomorphism notice that,
for $v, w \in [e, w_0]$, if $\phi(v) = \phi(w) s_j$ then
either $\phi(v) = \phi(R_j(w))$ or $\phi(v) = \phi(L_j(w))$,
and therefore either $v = R_j(w)$ or $v = L_j(w)$.
The converse implications are clear.
Thus, for $v, w \in [e, w_0]$,
\begin{eqnarray*}
& & \hbox{\rm $w$ covers $v$ in $\W(F_n)$} \\
&\Longleftrightarrow& v \in \{R_j(w), L_j(w) :\ 1 \le j \le n-2\} \quad
      \hbox{\rm and} \quad \rank(w) = \rank(v) + 1 \\ 
&\Longleftrightarrow& \phi(v) \in \{\phi(w)s_j :\ 1 \le j \le n-2\} \quad
      \hbox{\rm and} \quad \inv(\phi(w)) = \inv(\phi(v)) + 1 \\
&\Longleftrightarrow& \hbox{\rm $\phi(w)$ covers $\phi(v)$ in $\W(S_{n-1})$}.
\end{eqnarray*}

\end{proof}


We can now summarize.

\begin{theorem}\
\begin{itemize}
\item[1.]
The $0$-Hecke algebra $\HHH_{n-1}(0)$ acts faithfully on $\W(F_n)$.
\item[2.]
Each lower interval is invariant under this action.
\end{itemize}
\end{theorem}

\begin{proof}
By Lemma~\ref{t.down-relations}, the down operators  $\{ D_i :\ 1 \le i \le n-2\}$ 
satisfy the defining relations of $\HHH_{n-1}(0)$. 
Furthermore, by  Lemma~\ref{interval1}, 
lower intervals are invariant under these operators.  
By Theorem~\ref{t.main2}, the size of each maximal interval is equal to 
the dimension of the $0$-Hecke algebra, so that the action is faithful.

\end{proof}

\section{Number of maximal elements}
\label{section:proofs}

In this section we prove

\begin{theorem}\label{t.main3}
The number of maximal elements in $\W(F_n)$ is the Catalan number
$$
C_{n-1}=\frac{1}{n}{2n-2\choose n-1}.
$$
\end{theorem}

For a $q$-analogue of this result see Corollary~\ref{t.main3-q} below.

\bigskip

Recall the map $\phi: F_n\rightarrow S_{n-1}$ defined in Section~\ref{section:map}.

\begin{notation}
For $\pi\in S_{n-1}$, let
\[
N(\pi) := \left|\{w\in F_n\,|\,\phi(w) = \pi\}\right|.
\]
\end{notation}

For a subset $S$ of $\{1,\ldots,n-1\}$
let $\pi|_S$ be the subsequence of $(\pi(1),\ldots,\pi(n-1))$
consisting only of the values in $S$,
and let $p(\pi|_S)$ be the pattern (permutation) obtained by
mapping $S$ monotonically to $\{1, \ldots, |S|\}$, namely:
the sequence obtained from $\pi|_S$ by 
writing $1$ instead of the smallest number in $S$,
$2$ instead of the second smallest number, etc.
For example, if $\pi = 3572146 \in S_7$ and $S = \{1,2,6,7\}$
then $\pi|_S = 7216$ and $p(\pi|_S) = 4213$.

Also, for $1\le i<j\le n$, let
\[
A(i,j) := \{i, \ldots, j-2\} 
\]
(of size $j-i-1\ge 0$) and
\[
B(i,j) := \{1, \ldots, i-1\} \cup \{j, \ldots, n-1\} 
\]
(of size $n-j+i-1\ge 0$). Thus $A(i,j) \cup B(i,j) = \{1,\ldots,n\}\setminus \{j-1,n\}$,
a disjoint union.

\begin{lemma}\label{t.cover-number}
Let $\pi\in S_{n-1}$ 
such that $\pi(1) = j - 1$ $(2 \le j \le n)$. Then
\[
N(\pi) = \sum_{1\le i<j} N(p(\pi|_{A(i,j)})) \cdot N(p(\pi|_{B(i,j)})).
\]
\end{lemma}

\begin{example}
For $\pi = 41352 \in S_5$, $j-1 = \pi(1) = 4$.
Thus, summing over $1\le i\le 4$:
\begin{eqnarray*}
N(41352) &=& N(p(132))N(p(5)) + N(p(32))N(p(15)) + N(p(3))N(p(152)) +\\
                 & & N(p())N(p(1352))\\
         &=& N(132)N(1) + N(21)N(12) + N(1)N(132) + N()N(1342).
\end{eqnarray*}
\end{example}

\begin{proof}
Assume that $w\in F_n$ satisfies $\phi(w)(1) = j-1$.
Then the first transposition in $w$ has the form $t_1 = (i,j)$ for
some $i<j$, and the corresponding partial product 
$\sigma_2 = t_2 \cdots t_{n-1}$ is $(i,i+1,\ldots,j-1)(j,j+1,\ldots,n,1,\ldots,i-1)$.
In general, for any $1 \le k \le n-1$, $\sigma_k$ is a product of 
$k$ disjoint cycles, each of which forming a (not necessarily consecutive) 
cyclic subsequence of $\sigma_1 = (1, \ldots, n)$, and 
the transition from $\sigma_{k+1}$ to $\sigma_k$ amounts to
merging two of the cycles of $\sigma_{k+1}$ into one cycle
in which the original cycles form complementary cyclic intervals.
It follows that $(i,i+1,\ldots,j-1)$ is the product of those $t_k$
with $k>1$ which permute two elements of the set
$\{i,i+1,\ldots,j-1\}$, and thus have $\phi(w)(k)$ in this set
but of course $\phi(w)(k)\ne \phi(w)(1) = j-1$; and, similarly,
$(j,j+1,\ldots,n,1,\ldots,i-1)$ is the product of those $t_k$ with
$k>1$ which permute two elements of the set
$\{j,j+1,\ldots,n,1,\ldots,i-1\}$, and thus have $\phi(w)(k)$ in
this set but of course $\phi(w)(k)\ne n$. 
The word $w$ is a shuffle of these two sequences of transpositions, 
with the addition of $t_1 = (i,j)$. The $\phi$ values at these two sequences 
compose the ranges $A(i,j)$ and $B(i,j)$, respectively. 
Mapping these ranges monotonically to the ranges 
$\{1,\ldots,j-i-1\}$ and $\{1,\ldots,n-j+i-1\}$ 
we get $w_1 \in F_{j-i}$ and $w_2 \in F_{n-j+i}$ 
such that, for $\pi := \phi(w)$, 
$\phi(w_1) = p(\pi|_{A(i,j)})$ and $\phi(w_2) = p(\pi|_{B(i,j)})$.
Conversely, for any $1 \le i < j \le n$, $w_1 \in F_{j-i}$ and $w_2 \in F_{n-j+i}$
we can construct $w \in F_n$ by reversing the above process.
This completes the proof of the formula.

\end{proof}


\medskip

\noindent {\it Proof of Theorem~\ref{t.main3}.}
Let $\pi_0$ be the longest element in $S_{n-1}$, i.e., $\pi_0(i) = n-i$ $(1\le i\le n-1)$;
denote $N_n := N(\pi_0)$.
If $w\in F_n$ has $\phi(w) = \pi_0$ then $j = 1 + \pi_0(1) = n$ and therefore,
for each $1\le i\le n-1$,  $A(i,j) = \{i,\ldots,n-2\}$ and $B(i,j) = \{1,\dots,i-1\}$.
Thus $\pi_0|_{A(i,j)}$ and $\pi_0|_{B(i,j)}$ are decreasing sequences, and
the corresponding permutations $p(\pi_0|_{A(i,j)})$ and
$p(\pi_0|_{B(i,j)})$ are the longest elements in $S_{n-i-1}$ and
$S_{i-1}$, respectively. It follows, by Lemma~\ref{t.cover-number}, that
\[
N_n = \sum_{i=1}^{n-1} N_{n-i}N_{i}\qquad(\forall n\ge 2),
\]
and clearly $N_1 = 1$.
We recognize the recursion formula for the Catalan numbers $C_{n-1}$,
and therefore
\[
N_n = C_{n-1} = \frac{1}{n}{2n-2 \choose n-1} \qquad (n \ge 1),
\]
as claimed.


\qed



%
%
%

\section{Alternating non-crossing trees}
\label{section:alternating_trees}

The current section presents a bijection between the maximal elements in $\W(F_n)$
and alternating non-crossing trees, following a comment by an anonymous referee.
Recall that a tree with integer-labeled vertices is {\em alternating}
if the vertices along any path in it form an alternating sequence: $\ldots > i < j > k < \ldots$;
equivalently, if the neighbors of each vertex $i$ are all larger, or all smaller, than $i$.

\begin{observation}\label{t.max_to_alt}
If $v$ is a maximal element in $\W(F_n)$ then the corresponding non-crossing tree $G(v)$
(as in Proposition~\ref{GY22}) is alternating.
\end{observation}

\begin{proof}
If $v$ is maximal in $\W(F_n)$ then, by Lemma~\ref{interval2} for $\pi = id$,
$\phi(v) = \pi_0 = [n-1, \ldots, 1]$ and thus all pairs of letters in $\phi(v)$ are inversions.
For $a < b < c$, the pair $((a,b), (b,c))$  does {\em not} correspond to an inversion in $\phi(v)$,
by the characterization in Lemma~\ref{t.inversions}. This implies that $G(v)$ is alternating.

\end{proof}

The number of all alternating non-crossing trees on the set of vertices $[n] = \{1, \ldots, n\}$
was shown~\cite[Theorem 6.4]{GGP} to be the Catalan number $C_{n-1}$.
This motivates the following result.

\begin{theorem}\label{t.max_eq_alt}
The mapping $v \mapsto G(v)$ is a bijection between the set of all maximal elements in $\W(F_n)$ and
the set of all alternating non-crossing trees on the set of vertices $[n]$.
\end{theorem}

\begin{proof}
By Observation~\ref{t.max_to_alt}, this mapping, to be denoted $f$, is well defined.
Let us construct a map $g$ in the other direction.

For an alternating non-crossing tree $G$ on $[n]$, defining $v = g(G)$ 
amounts to  imposing a linear order on the $n-1$ edges of $G$.
This can be done recursively as follows: If $n = 1$, there are no
edges and the order is empty. For $n \ge 2$ note that $\{1,n\}$ must
be an edge of $G$ (since the path from $1$ to $n$ in $G$ is
monotone increasing, because of the non-crossing property, and
thus is also alternating if and only if it consists of a single
edge). The graph obtained from $G$ by deleting the edge $\{1,n\}$
has two connected components, $G_1$ (containing the vertex $1$)
and $G_n$ (containing the vertex $n$). They are alternating
non-crossing trees on complementary subsets $[1,k]$ and
$[k+1,n]$ of $[n]$, for a suitable $1 \le k \le n-1$. They define,
by recursion, corresponding words $v_1 = g(G_1)$ and $v_n =
g(G_n)$ (using the same letter $g$ for the two corresponding maps). 
Define the word $v = g(G)$ as the word $v_n$, followed by
the letter $(1,n)$, followed by the word $v_1$.

We need to show that $v$ is a maximal element of $\W(F_n)$.
First, $v \in F_n$ since $G = G(v)$ satisfies the Goulden-Yong conditions of Proposition~\ref{GY22},
the cyclic order on neighbors following recursively from the construction.
Maximality of $v$ in $\W(F_n)$ follows from the fact that all pairs in the permutation $\phi(v)$
are inversions, by a careful application of Lemma~\ref{t.inversions}.

Having defined the map $g$, it is clear that the composition $f \circ g$ is
the identity map on the set of all alternating non-crossing trees.
Thus $f$ is surjective, and since both sets have the same (Catalan) size
it follows that $f$ is a bijection.

\end{proof}

\begin{corollary}\label{1n}
Every maximal element in $\W(F_n)$ contains $(1,n)$ as a factor.
\end{corollary}

\begin{proof}
This follows from Theorem~\ref{t.max_eq_alt} together with the observation 
(mentioned in this theorem's proof) that 
every alternating non-crossing tree of order $n\ge 2$ contains the edge $\{1,n\}$.

\end{proof}

\section{Right and left inversions}
\label{section:problems}



Recall that the weak order on $S_n$ is characterized by inversion sets, namely:
$\pi\le \sigma$ in the (right) weak order on $S_n$ if and only if
$\Inv(\pi)\subseteq \Inv(\sigma)$, where
$\Inv(\pi):=\{(i,j):\ i<j\ \text {and}\  \pi^{-1}(i)>\pi^{-1}(j)\}$;
see, e.g., \cite[Cor.\ 1.5.2 and Prop.\ 3.1.3]{BB}.


The goal of this section is to provide a similar characterization for $\W(F_n)$.

\begin{defn}\label{defn-LR-inversions}\ \rm
The set of {\em inversions} of $w=(t_1,\dots, t_{n-1})\in F_n$  is
$\Inv(w):=\Inv(\phi(w))$. 
A pair $(i,j)\in \Inv(\phi(w))$ is
\begin{itemize}
\item
a {\em right inversion} 
if there exist $a<b<c$ such that $t_{\phi(w)^{-1}(j)}=(a,c)$ and
$t_{\phi(w)^{-1}(i)}=(a,b)$, or there exist $a<b<c<d$ such that
$t_{\phi(w)^{-1}(j)}=(a,d)$ and $t_{\phi(w)^{-1}(i)}=(b,c)$;
\item
a {\em left inversion} 
if there exist $a<b<c$ such that
$t_{\phi(w)^{-1}(j)}=(b,c)$ and $t_{\phi(w)^{-1}(i)}=(a,c)$, or
there exist $a<b<c<d$ such that $t_{\phi(w)^{-1}(j)}=(b,c)$ and
$t_{\phi(w)^{-1}(i)}=(a,d)$;
\item
a {\em neutral inversion} 
if there exist $a<b<c<d$ such that $t_{\phi(w)^{-1}(j)}=(c,d)$ and
$t_{\phi(w)^{-1}(i)}=(a,b)$.
\end{itemize}
\end{defn}

\begin{remark}\label{left-right-explanation}\ \rm
If $t_{\phi(w)^{-1}(i)}=(a,b)$, $a<b$, write $I_i$ for the
interval $[a,b]\subseteq [1,n]$. Then $(i,j)\in \Inv(\phi(w))$ is
\begin{itemize}
\item
a right inversion if and only if $I_i\subset I_j$;
\item
a left inversion if and only if $I_j\subset I_i$;
\item
a neutral inversion if and only if $I_i\cap I_j=\emptyset$.
\end{itemize}
\end{remark}

Denote the set of right, left and neutral inversions of $w\in F_n$
by $\Inv_R(w)$, $\Inv_L(w)$ and $\Inv_N(w)$, respectively.
Lemma~\ref{t.inversions} implies
\begin{observation}\label{t.LN}
For every $w\in F_n$,
\[
\Inv(w)= 
\Inv_R(w)\cup \Inv_L(w)\cup \Inv_N(w),
\]
a disjoint union.
\end{observation}

\begin{example}
For $w=((3,4),(1,5),(5,6),(1,4),(1,2))\in F_6$, 
$\phi(w)=[3,4,5,2,1]$. 
Then $\Inv(w)=\{(1,2),(1,3),(1,4),(1,5),(2,3),(2,4),(2,5)\}$ 
is a disjoint union of 
$\Inv_R(w)=\{(1,2),(1,4),(2,4)\}$, \,
$\Inv_N(w)=\{(1,3), (1,5), (2,5)\}$ \, and \,
$\Inv_L(w)=\{(2,3)\}$.
\end{example}


Define the involution $\iota: F_n\longrightarrow F_n$ by 
\[
\iota(t_1,\ldots,t_{n-1}) := (t_{n-1}^{\sigma_0}, \ldots ,t_1^{\sigma_0}),
\]
where $g^h:= h^{-1}gh$ and $\sigma_0:=[n,\ldots,1] \in S_n$.
In Figure~\ref{figure2-tik}, for example, $\iota$ corresponds to a reflection of the Hasse diagram
of $\W(F_4)$ in a vertical line through $e$.

\begin{observation}\label{iota}
Let $\pi_0 = [n-1, \ldots, 1] \in S_{n-1}$.
\begin{itemize}
\item[1.]
For every $w\in F_n$:
$\phi(\iota(w))=\phi(w)^{\pi_0}$ and also
$\Inv_N(\iota(w)) = \Inv_N(w)^{\pi_0}$,
$\Inv_R(\iota (w)) = \Inv_L(w)^{\pi_0}$ and
$\Inv_L(\iota (w))=\Inv_R(w)^{\pi_0}$.
\item[2.]
$\iota$ is a rank preserving automorphism of the poset $\W(F_n)$.
\end{itemize}
\end{observation}


\begin{proposition}\label{determine}\ \\
Each of the pairs $(\Inv(u), \, \Inv_{R}(u))$ and
$(\Inv(u), \, \Inv_{L}(u))$  determines the word $u \in F_n$ uniquely;
namely, for $u,v\in F_n$,
\begin{eqnarray*}
u = v  
\ &\iff& \
\Inv(u) = \Inv(v) \ \ \text{\rm and} \ \ \Inv_{R}(u) = \Inv_{R}(v) \\ 
\ &\iff& \ 
\Inv(u) = \Inv(v) \ \ \text{\rm and} \ \ \Inv_{L}(u) = \Inv_{L}(v).
\end{eqnarray*}
\end{proposition}

\begin{proof}
If $u = v$ then, of course, all the other equalities hold.
Let us prove that if $\Inv(u) = \Inv(v)$ and $\Inv_{L}(u) = \Inv_{L}(v)$
then $u = v$.
We shall use the shorthand notation $\phi_u := \phi(u) \in S_{n-1}$
for $u = (t_1, \ldots, t_{n-1}) \in F_n$.
Of course, the inversion set $\Inv(u) = \Inv(\phi_u)$ determines $\phi_u$ uniquely
since it determines, for any $i \ne j$, whether or not $\phi_u^{-1}(i) < \phi_u^{-1}(j)$.
we shall show, by induction on $n$, that $\Inv(u)$ and $\Inv_L(u)$ determine
$t_{r}$ for all $1 \le r \le n-1$.

This trivially holds for $n \le 2$.
For the induction step assume that $t_{n-1} = (i,j)$ with $i<j$. 
By Observation~\ref{obs.phi}, $\phi_u(n-1) = i$. 
We can recover $j$ from $\Inv_L(u)$ once we show that there are
exactly $j-i-1$ left inversions of $\phi_u$ involving $\phi_u(n-1)$.
Indeed, the inversions involving $\phi_u(n-1)$ are exactly
the ordered pairs $(\phi_u(n-1), \phi_u(k))$ such that
$k < n-1$ (automatically satisfied) and $\phi_u(k) > \phi_u(n-1)$.
By Remark~\ref{left-right-explanation}, such a pair is a {\em left} inversion
if and only if the $k$-th edge of $G(u)$ has
both endpoints in the interval $I_{\phi_u(n-1)} = [i,j]$, 
and is not the edge $\{i,j\}$ itself. 
By Proposition~\ref{GY22}, there are exactly $j-i-1$ such edges.
Thus we have recovered both $i$ and $j$.

Now note that 
$t_1 \cdots t_{n-2} = c \cdot (i,j)$ is a product of two disjoint cycles, 
$c' = (i+1, \ldots, j)$ and $c'' = (j+1, \ldots, i) = (1, \ldots, i, j+1, \ldots, n)$,
of corresponding lengths $n' = j-i$ and $n'' = n-j+i$. 
Denote $V' := \{i+1, \ldots, j\}$ and $V'' := \{1, \ldots, i, j+1, \ldots, n\}$,
and let $s': \{1, \ldots, n'\} \to V'$ and $s'': \{1, \ldots, n''\} \to V''$ be
monotone increasing bijections.
The word $(t_1, \ldots, t_{n-2})$ is a shuffle of two subwords, 
one supported in $V'$ and the other in $V''$.
For $1 \le k \le n-2$ we don't yet know $t_k$, but we do know
whether it is supported in $V'$ or in $V''$ -- depending on whether
$t_{n-1}(\phi_u(k))$ belongs to $V'$ or to $V''$.
Let $K' := \{k'_1, \ldots, k'_{n'-1}\}$  and $K'' := \{k''_1, \ldots, k''_{n''-1}\}$ 
be the sets of indices $k$ for which $t_k$ is supported in $V'$ and in $V''$, respectively.  
Denoting $t'_{k'} := (s')^{-1} t_{k'} s'$, the word 
$u' := (t'_{k'_1}, \ldots, t'_{k'_{n'-1}})$ is in $F_{n'}$, 
and a similar definition gives $u'' \in F_{n''}$.
By induction, it suffices to show that $\Inv(u)$ and $\Inv_L(u)$ determine 
$\Inv(u')$ and $\Inv_L(u')$ (since $t_{k'}$ can then be retrieved from $t'_{k'}$), 
and similarly for $u''$.

Actually, since $s'$ is monotone increasing, $s'(\phi_{u'}(t)) = t_{n-1}(\phi_u(k'_t))$
for any $1 \le t \le n'-1$. Since $\phi_{u'}(t) \ne n'$ and $i \not\in V'$, it follows that
$s'(\phi_{u'}(t)) \not\in \{i, j\}$ and therefore $s'(\phi_{u'}(t)) = \phi_u(k'_t)$.
Thus there is a bijection between the inversions of $u'$ and the relevant inversions of $u$,
which may be written (with a slight abuse of notation) as $s'(\Inv(u')) = \Inv(u|_{K'})$.
The situation for $u''$ is only slightly different, since indeed $i \in V''$ but 
an application of $t_{n-1}$, exchanging $i$ with $j$, does not change 
its relative position in the set. Therefore $s''(\Inv(u'')) = t_{n-1}(\Inv(u|_{K''}))$.
Remark~\ref{left-right-explanation} now shows that similar relations hold
for $\Inv_L$, completing the induction step.
Thus $\Inv(u)$ and $\Inv_L(u)$ uniquely determine $u$.


Finally, by Observation~\ref{iota}.1, 
$\Inv(u)$ and $\Inv_R(u)$ determine
$\Inv(\iota(u))$ and $\Inv_L(\iota(u))$, thus $\iota(u)$ and $u$.

\end{proof}

\begin{remark}
$\Inv_L(u)$ and $\Inv_R(u)$ alone do not suffice to uniquely determine $u \in F_n$.
Here is an example from $F_4$ (see Figure~\ref{figure2-tik}):
$u_1 = (12, 34, 24)$ and $u_2 = (34, 12, 24)$ are distinct, but have 
$\Inv_L(u_1) = \Inv_L(u_2) = \{(2, 3)\}$ as well as 
$\Inv_R(u_1) = \Inv_R(u_2) = \emptyset$
(while $\Inv_N(u_1) = \emptyset \ne \{(1, 3)\} = \Inv_N(u_2)$).
\end{remark}

\medskip


\begin{proposition}\label{criterion}
For $u, v \in F_n$, if $u \le v$ in $\W(F_n)$ then
$\Inv(u) \subseteq \Inv(v)$,  $\Inv_R(u) \subseteq \Inv_R(v)$
and $\Inv_L(u) \subseteq \Inv_L(v)$.
\end{proposition}

\begin{proof}
It suffices to prove the claim when $v$ covers $u$ in $\W(F_n)$. 
Thus $\phi(u) = \phi(v) s_i < \phi(v)$ in $S_{n-1}$, for some $1 \le i \le n-2$,
and consequently $\Inv(u) = \Inv(v) \setminus \{p_0\} \subseteq \Inv(v)$
for the pair $p_0 = (\phi(v)(i+1), \phi(v)(i)) = (\phi(u)(i), \phi(u)(i+1))$.
There are three possible cases:
$u = R_i(v) = L_i(v)$, $u = R_i(v) \ne L_i(v)$, and $u = L_i(v) \ne R_i(v)$.
We want to show that 
$\Inv_R(u) \subseteq \Inv_R(v)$ and $\Inv_L(u) \subseteq \Inv_L(v)$.

Consider an arbitrary inversion $p = (\phi(u)(k), \phi(u)(j))$ of $u$.
It is also an inversion of $v$.
If $j \not\in \{i, i+1\}$ then $t_j(u) = t_j(v)$, where 
$u = (t_1(u), \ldots, t_{n-1}(u))$ and similarly for $v$. 
By Remark~\ref{left-right-explanation}, the type of $p$
(left, right or neutral) is then the same in $u$ as in $v$. 

If $u = R_i(v) = L_i(v)$, namely if $t_i(v)$ and $t_{i+1}(v)$ commute,
then $t_i(u) = t_{i+1}(v)$ and $t_{i+1}(u) = t_i(v)$,
and therefore {\em all} inversions keep their type upon transforming $u$ to $v$:
$\Inv_R(u) = \Inv_R(v) \setminus \{p_0\}$ and
$\Inv_L(u) = \Inv_L(v) \setminus \{p_0\}$.

Assume now that $u = R_i(v) \ne L_i(v)$. Then $t_i(v)$ and $t_{i+1}(v)$
do not commute, and by Lemma~\ref{t.inversions} there exist $a < b < c$
such that $(t_i(v), t_{i+1}(v)) = ((b,c), (a,c))$ and $(t_i(u), t_{i+1}(u)) = ((a,b), (b,c))$.
Since $t_{i+1}(u) = t_i(v)$, the inversions involving $\phi(u)(i+1)$ do not
change their type upon transforming $u$ to $v$. It remains to consider
the inversions involving $\phi(u)(i)$. Consider the graphs $G(v)$ and $G(u)$,
which are identical except that the former has an edge $\{a, c\}$
and the latter an edge $\{a, b\}$ (and both have an edge $\{b, c\}$).
For any other edge of these graphs, the non-crossing property implies that
the corresponding interval $I_k$ belongs to one of the following four classes:
Contained in $[b,c]$; contained in $[a,b]$; containing $[a,c]$; and 
disjoint from $[a,c]$. 
We consider only edges corresponding to inversions with $\phi(u)(i)$.
Using Remark~\ref{left-right-explanation} again we see that,
for the last three classes, transforming $u$ to $v$
(namely replacing $\{a, b\}$ by $\{a, c\}$) does not change the type of
inversion with $\phi(u)(i)$. For the first class, the type changes
from neutral to either right or left.
Thus we conclude that 
$\Inv_R(u) \subseteq \Inv_R(v)$ and $\Inv_L(u) \subseteq \Inv_L(v)$.

The treatment of the case $u = L_i(v) \ne R_i(v)$, for which
$(t_i(v), t_{i+1}(v)) = ((a,c), (a,b))$ and $(t_i(u), t_{i+1}(u)) = ((a,b), (b,c))$,
is similar.

\end{proof}

\begin{remark}
Various natural stronger versions of Proposition~\ref{criterion} 
can be shown to be false by examples from $F_4$.
For example, $\Inv_N(u)$, $\Inv_R(u) \cup \Inv_N(u)$ and $\Inv_L(u) \cup \Inv_N(u)$
are not even weakly monotone increasing functions of $u$;
and inclusions of any two of the sets $\Inv$, $\Inv_R$ and $\Inv_L$
does not imply $u \le v$ in $\W(F_n)$.
\end{remark}

It would be interesting to know whether the converse of Proposition~\ref{criterion}
holds.

\section{Right inversions and $q$-Catalan numbers}

A $q$-analogue of Theorem~\ref{t.main3} is given in this section.

\subsection{Enumeration of words by right and left inversions}

Carlitz and Riordan~\cite{CR} defined a $q$-Catalan number $C_n(q)$ using the recursion
$$
C_{n+1}(q):=\sum\limits_{k=0}^n q^{(k+1)(n-k)} C_k(q) C_{n-k}(q)\qquad(n\ge 0)
$$
with $C_0(q):=1$. For combinatorial interpretations of this number
and further information see, e.g., \cite{FH, Butler, Sagan-Savage}
and sequence A138158 in~\cite{Sloane}.
In particular, $C_{n}(q)$ counts Dyck paths of length $n$
with respect to the area above the path;
see, e.g., \cite{FH}.


\begin{defn}\label{def:qt-catalan}
Define {\em $(q,t)$-Catalan numbers} by
$$
\tC_{n+1}(q,t):=\sum\limits_{k=0}^n q^k t^{n-k} \tC_k(q,t) \tC_{n-k}(q,t)\qquad(n\ge 0)
$$
with
$$
\tC_{0}(q,t) = 1.
$$
\end{defn}

\begin{observation}\label{obs.qt2}
For every $n\ge 0$,
\begin{equation}\label{qt-eq2}
\tC_{n}(q,t)=\tC_{n}(t,q)
\end{equation}
and
\begin{equation}\label{qt-eq1}
C_n(q)= q^{{n\choose 2}} \tC_n(q^{-1},1).
\end{equation}
\end{observation}

\bigskip

\begin{defn}
For $w\in F_n$, denote by $\inv_R(w)$ the cardinality of $\Inv_R(w)$ and
by $\inv_L(w)$ the cardinality of $\Inv_L(w)$.
Let $\max(F_n)$ denote the set of all maximal elements in $F_n$.
\end{defn}

\begin{proposition}\label{qt-prop}
For every positive integer $n$,
\[
\sum\limits_{w \in \max(F_{n+1})} q^{\inv_R(w)} t^{\inv_L(w)} = \tC_{n}(q,t).
\]
\end{proposition}

\begin{proof}
By induction on $n$. For $n=1$ the claim obviously holds.

For the induction step recall that, by Corollary~\ref{1n},  any
maximal element in $w$ in $\W(F_{n+1})$ contains the factor
$(1,n+1)$.
Denote by $F_{n+1,k}$ the set of maximal elements in $\W(F_{n+1})$
whose $(k+1)$-st factor is $(1,n+1)$.
By the proof of Theorem~\ref{t.main3}, this set is characterized as follows:
$w=(t_1,\ldots,t_n) \in F_{n+1,k}$ if and only if
$t_1 t_2\cdots t_{k}$ is a reduced word for the cycle $(n-k,\dots, n)$,
where $\phi(t_1 t_2\cdots t_{k})=[n-1,\dots,n-k]$,
and $t_{k+2} t_{k+3}\cdots t_{n}$ is a
reduced word for the cycle $(1,2,\dots,n-k)$,
where $\phi(t_{k+2} t_{k+3}\cdots t_{n})=[n-k-1, n-k-2,\dots,1]$. 

\smallskip

For every $w=t_1\cdots t_n \in F_{n+1,k}$ the following holds: for
every $1\le i\le k$ there exist $n-k\le c< d\le n$ such that
$t_i=(c,d)$; for every $k+1< j\le n$ there exist $1\le a<b< n-k$
such that $t_j=(a,b)$. Thus, by
Definition~\ref{defn-LR-inversions},
all pairs $(t_i,t_j)$, with $i\le k$ and $j>k+1$,
 are neutral inversions; all pairs $(t_i, t_{k+1})$, $1\le i\le k$, are left
 inversions and all pairs $(t_{k+1}, t_j)$, $k+1<j\le n$, are
 right inversions.

\smallskip


This implies that, for every $0\le k\le n$,
\begin{eqnarray*}
& & \sum\limits_{w \in F_{n+1,k}} q^{\inv_R(w)} t^{\inv_L(w)} \\
&=& q^{n-k-1} t^k\cdot \sum\limits_{w \in \max(F_{k+1})} q^{\inv_R(w)}t^{\inv_L(w)} \cdot
\sum\limits_{w \in \max(F_{n-k+1})} q^{\inv_R(w)}t^{\inv_L(w)} \\
&=& q^{n-k-1} t^k \cdot \tC_k (q,t) \cdot \tC_{n-k} (q,t).
\end{eqnarray*}
The last equality follows from the induction hypothesis. Thus
\begin{eqnarray*}
\sum\limits_{w\in \max(F_{n+1})} q^{\inv_R(w)}t^{\inv_L(w)}
&=& \sum\limits_{k=0}^n \sum\limits_{w\in F_{n+1,k}} q^{\inv_R(w)}t^{\inv_L(w)} \\
&=& \sum\limits_{k=0}^n  q^{n-k} t^k\cdot \tC_k (q,t)\cdot \tC_{n-k} (q,t)=\tC_{n+1}(q,t).
\end{eqnarray*}
The last equality follows from the defining recursion of $\tC_n(q,t)$
together with (\ref{qt-eq2}).

\end{proof}



\begin{corollary}\label{t.main3-q}
For every positive integer $n$,
\[
\sum\limits_{w\in \max(F_{n+1})} q^{\inv_R(w)} = q^{n\choose 2} C_{n}(q^{-1})
\]
or equivalently
\[
\sum\limits_{w \in \max(F_{n+1})} q^{\inv_{LN}(w)} = C_{n}(q).
\]
\end{corollary}

\begin{proof}
Proposition~\ref{qt-prop} together with (\ref{qt-eq1}).

\end{proof}


\subsection{Refined enumeration of alternating non-crossing trees}

Denote the set of alternating non-crossing trees of order $n$ by
${\mathcal T}_n$.

\begin{defn}\label{edge_pairs}
A pair of distinct edges $\{e, f\}$ in a non-crossing alternating tree $T$ is
\begin{itemize}
\item
{\em neutral} if $e = \{c,d\}$ and $f = \{a,b\}$ for some $a < b < c < d$;
\item
{\em right} if $e = \{a,d\}$ and $f = \{b,c\}$ for some  $a \le b < c < d$ and
the edge $\{b,c\}$ is in the component of $a$ in $T \setminus \{\{a,d\}\}$ 
(for example, whenever $a = b$); and
\item
{\em left} if $e = \{b,c\}$ and $f = \{a,d\}$ for some $a < b < c \le d$ and
the edge $\{b,c\}$ is in the component of $d$ in $T \setminus \{\{a,d\}\}$ 
(for example, whenever $c = d$).
\end{itemize}
\end{defn}

\medskip

\noindent Denote the cardinality of right and left and pairs in an
alternating non-crossing tree $T\in {\mathcal T}_n$ by $\rp(T)$
and $\lp(T)$, respectively.

\begin{proposition}\label{qt-prop-trees}
For every positive integer $n$,
\[
\sum\limits_{T \in {\mathcal T}_n} q^{\rp(T)} t^{\lp(T)}=
\tC_{n-1}(q,t).
\]
\end{proposition}

\begin{proof}
Each pair of distinct edges is either neutral, right or left.
Moreover, the edge $e$ precedes $f$ in the canonical linear order
on $T$ defined by the map $g$ as in the proof of
Theorem~\ref{t.max_eq_alt}.  
Hence, a pair of edges in an alternating non-crossing tree is a
right (left) pair if and only if the corresponding letters in the
word $v = g(G)$ form a right (left) inversion. It follows that
\[
\sum\limits_{T \in {\mathcal T}_n} q^{\rp(T)} t^{\lp(T)}=
\sum\limits_{w \in \max(F_{n})} q^{\inv_R(w)} t^{\inv_L(w)}.
\]
Proposition~\ref{qt-prop} completes the proof.
\end{proof}

\subsection{Another combinatorial interpretation}

Consider the set $D(n)$ of Dyck paths of length $2n$.
Each $p\in D(n)$ is a sequence $(\e_1,\ldots,\e_{2n})$,
with
$$
\e_i \in \{1,-1\}\qquad(\forall i),
$$
such that
$$
\e_1 + \ldots + \e_{2n} = 0
$$
and
$$
\e_1 + \ldots + \e_{k} \ge 0 \qquad(\forall k).
$$
Geometrically, $p$ corresponds to a lattice path from $(0,0)$ to $(2n,0)$,
with admissible steps $(1,1)$ and $(1,-1)$,
which never goes below the horizontal line ({\em baseline}, in the sequel) connecting its endpoints.
\begin{defn}
With each $p\in D(n)$ associate the following two statistics:
\begin{eqnarray*}
\area(p)
&:=& {n \choose 2} - \#\{(i,j)\,|\,1\le i < j\le 2n,\,-\e_i=\e_j=1\}\\
&\,=& -{n+1 \choose 2} + \#\{(i,j)\,|\,1\le i < j\le 2n,\,\e_i=-\e_j=1\},\\
\bmaj(p)
&:=& \sum_{i\,:\,-\e_i=\e_{i+1}=1} \max\{k\,|\,(\forall 1\le t\le 2k)\,\e_{i+1}+\ldots+\e_{i+t}\ge 0\}.
\end{eqnarray*}
\end{defn}

Geometrically, $\area(p)$ is the area between the path $p$ and the
``horizontal'' path $p_0=(1,-1,1,-1,...)$, where the square with
vertices $(0,0)$, $(1,1)$, $(2,0)$ and $(1,-1)$ is considered to
have area $1$; while $\bmaj(p)$ is the sum, over all internal
local minima of the path $p$, of half the length of the maximal
horizontal segment starting at the minimum point and stretching to
the right without crossing (but possibly touching) $p$. Note that
$\area(p)$ is (a linear function of) the inversion number of the
sequence $p$.



\begin{proposition}\label{dyck}
For every positive integer $n$
$$
\sum\limits_{p\in D(n)} q^{\area(p)} t^{\bmaj(p)}= \tC_n(q,t).
$$
\end{proposition}

Note that for $t=1$ this is a classical result~\cite{FH}.

\begin{proof}
It suffices to show that the LHS, to be denoted $C'_n(q,t)$,
satisfies the recursion in Definition~\ref{def:qt-catalan};
it clearly satisfies $C'_0(q,t) = 1$.

Let $p\in D(n+1)$, and define $k := \min\{t\ge 0\,|\,\e_1+\ldots+\e_{2t+2} = 0\}$,
so that $2k+2$ is the first time that $p$ returns to the baseline. Clearly $0\le k\le n$.
We can write $p$ as the concatenation of sequences
$$
p = (1) * p' * (-1) * p'',
$$
where $p'\in D(k)$ and $p''\in D(n-k)$. Now clearly
$$
\area(p) = k + \area(p') + \area(p'')
$$
and
$$
\bmaj(p) = (n-k) + \bmaj(p') + \bmaj(p''),
$$
since the path $p$ has an internal local minimum at $(2k+2,0)$,
in addition to the internal local minima of $p'$ and $p''$.
Thus
\begin{eqnarray*}
C'_{n+1}(q,t) = \sum_{k=0}^{n} q^k t^{n-k} C'_k(q,t) C'_{n-k}(q,t),
\end{eqnarray*}
and this completes the proof.
\end{proof}

\begin{corollary}
$\area$ and $\bmaj$ are equidistributed:
$$
\sum\limits_{p\in D(n)} q^{\area(p)} =
\sum\limits_{p\in D(n)} q^{\bmaj(p)} =
q^{n \choose 2} C_n(q^{-1})
$$
\end{corollary}

\begin{proof}
Proposition~\ref{dyck} together with both parts of Observation~\ref{obs.qt2}.
\end{proof}

\begin{corollary}
$$
\sum\limits_{\{w\in F_{n+1}:\ \rank(w)={n\choose 2}\}}
q^{\inv_R(w)} t^{\inv_L(w)}=\sum\limits_{p\in D(n)} q^{\area(p)}
t^{\bmaj(p)}.
$$
\end{corollary}

\begin{proof}
Proposition~\ref{qt-prop} together with Proposition~\ref{dyck}.
\end{proof}

\bigskip

\section{Radius and diameter}%
\label{section:diameter}

\begin{defn}
Let $G = (V,E)$ be a finite connected undirected graph, and let $d(\cdot,\cdot)$
be the corresponding metric on its vertex set $V$.
The {\em radius} of $G$ is
\[
\radius(G) := \min_{v \in V} \max_{w \in V} d(v,w)
\]
and its {\em diameter} is
\[
\diameter(G) := \max_{v,w \in V} d(v,w).
\]
\end{defn}

Recall the Hurwitz graph $G_T(n)$ (Definition~\ref{d.Hurwitz_graph}), and let
$\Cayley(S_{n-1})$ be the (right) Cayley graph of the group $S_{n-1}$ with respect to its
standard Coxeter generators.

\begin{lemma}\label{t.phi_contracts}
The map $\phi : G_T(n) \to \Cayley(S_{n-1})$ is a contraction:
\[
d(\phi(v), \phi(w)) \le d(v,w)\qquad(\forall v, w \in F_n).
\]
\end{lemma}

\begin{proof}
It suffices to show that if $d(v,w) = 1$ then $d(\phi(v), \phi(w)) \le 1$.
Indeed, if $d(v,w) = 1$ in $G_T(n)$
then either $v = R_j(w)$ or $v = L_j(w)$ for some $j$.
Thus, by Lemma~\ref{t.main-lemma}, $\phi(v) \in \{\phi(w), \phi(w) s_j\}$,
so that indeed $d(\phi(v), \phi(w)) \le 1$.

\end{proof}

\begin{theorem}
The radius of the Hurwitz graph
\[
\radius(G_T(n)) = {n-1 \choose 2}.
\]
\end{theorem}

\begin{proof}
By Definition~\ref{d.order},  only edges connecting vertices of the same rank are deleted
during the passage from $G_T(n)$ to the Hasse diagram of $\W(F_n)$.
It follows that the distance $d(e,w)$
from the unique minimal element $e$ of $\W(F_n)$ to an arbitrary element $w \in F_n$
is the same, whether measured in $G_T(n)$ or in the Hasse diagram.
Thus, by Theorem~\ref{t.main2} and a well-known property of $S_{n-1}$,
\[
\max_{w \in F_n} d(e,w) = \max_{\pi \in S_{n-1}} d(id,\pi) = {n-1 \choose 2}.
\]
It follows that
\[
\radius(G_T(n)) \le {n-1 \choose 2}.
\]

Consider now an arbitrary $v \in F_n$.
By the proof of Theorem~\ref{t.main2}, $\phi : F_n \to S_{n-1}$ is surjective.
Thus there exists $w \in F_n$ with
\[
\phi(w) = \phi(v) \cdot \pi_0,
\]
where $\pi_0 = [n-1, \ldots, 1] \in S_{n-1}$.
For this $w$ it follows from Lemma~\ref{t.phi_contracts} that
\[
d(v,w) \ge d(\phi(v),\phi(w)) =  {n-1 \choose 2},
\]
and therefore
\[
\radius(G_T(n)) \ge {n-1 \choose 2}.
\]
\end{proof}

Clearly, for any graph $G$,
\[
\radius(G) \le \diameter(G) \le 2 \radius(G).
\]

\begin{corollary}
The diameter of $G_T(n)$ satisfies
\[
{n-1 \choose 2} \le \diameter(G_T(n)) \le 2{n-1 \choose 2}.
\]
\end{corollary}


The upper bound can be improved.

\begin{proposition}\label{t.diameter}
The diameter of $G_T(n)$ satisfies
\[
{n-1 \choose 2} \le \diameter(G_T(n)) \le \frac{3}{2}{n-1 \choose 2}.
\]
\end{proposition}

\begin{proof}
We shall use the ``bubble sort" algorithm to transform
any two given words $v, w \in F_n$ to a common word
by applying various $R_j$ and $L_j$ operators.

For the initial step note that, since every tree has at least one leaf
(vertex of degree $1$), by Proposition~\ref{GY22}(i) applied to $v$
there is an $1 \le i \le n$ which appears in exactly one of the 
$\ell = n - 1$ factors (transpositions) in $v$.
This unique ``$i$-factor" is within distance $\lceil \ell / 2 \rceil - 1$
from either the first or the last factor (``end factor") in $v$.
Apply a sequence of $R_j$ and $L_j$ operators to ``push" 
this $i$-factor towards the closest end factor without creating 
any additional $i$-factors.
This takes at most $\lceil \ell / 2 \rceil - 1$ steps, resulting in a word
whose only $i$-factor is an end factor.

On the other hand, consider the word $w$: It may contain more than
one $i$-factor, and we want to push all of them toward the end factor
(at the same end as in $v$). This can be done by pushing the $i$-factor
which is furthest from that end factor. At each step the furthermost
$i$-factor gets closer to the end factor, and the total number of $i$-factors
is unchanged or is reduced by $1$. After at most $\ell - 1$ steps we get
from $w$ a word whose only $i$-factor is the end factor. 

Now observe that if the only occurence of $i$ in a decomposition 
of $\pi \in S_n$ ($\pi = [n, \ldots, 1]$ for this initial step) 
is in the first (leftmost) factor then this factor is $(i, \pi(i))$; 
and if the only occurence is in the last factor then this factor is $(i, \pi^{-1}(i))$. 
This implies that the end $i$-factors in the two words obtained from
$v$ and $w$ are equal. 

For the rest of the algorithm we keep this end factor fixed, 
and work only on the subwords containing the other factors.
We repeat the procedure for the two words of length $n-2$
thus obtained from $v$ and $w$.
Obviously, after $n-2$ such steps we get two words of length $1$,
which are necessarily equal. The total number of operators used is at most
\[
\sum_{\ell = n-1}^{2} (\ell-1 + \lceil {\ell/2} \rceil - 1)
= {n-1 \choose 2} + \left \lfloor \frac{(n-2)^2}{4} \right \rfloor \le \frac{3}{2} {n-1 \choose 2},
\]
and this is an upper bound on the diameter of $G_T(n)$.

\end{proof}

\section{Final remarks and open problems}%
\label{section:open_problems}

An improvement of Proposition~\ref{t.diameter}, bounding 
the diameter of the Hurwitz graph, is desired.

\begin{conjecture}
The diameter of $G_T(n)$ is ${n-1\choose 2}+ O(n)$.
\end{conjecture}

Computations (using SAGE) for $n<8$ seem to suggest that the diameter
is actually $\left\lfloor (n-1)^2 / 2 \right\rfloor -1$.

\bigskip


As noted by an anonymous referee, some of the concepts and questions
presented in this paper may be generalized to other Coxeter groups.

\begin{defn}\label{d.Hurwitz_graph-Coxeter}
Let $W$ be a Coxeter group, $T$ its set of reflections, and $c$ a
Coxeter element in $W$. The {\em Hurwitz graph} $H(W)$ is the graph
with vertices corresponding to all reduced words expressing $c$ in the
alphabet $T$ of all reflections, where adjacency is determined by
a local right or left shift: Two words are adjacent in the graph if they
agree in all but two adjacent letters, which are $(s, t)$ in one word
and either $(t^s, s)$ or $(t, s^t)$ in the other.
\end{defn}

By a result of Bessis~\cite[Prop.\ 1.6.1]{Bessis}, 
the Hurwitz graph of every finite Coxeter group is connected.

\begin{question}
Find the radius and diameter of $H(W)$.
\end{question}

It should be noted that, for every Coxeter group $W$, 
$H(W)$ is a Schreier graph of the braid group of type $A$.


\begin{conjecture}
The radius of $H(B_n)$ is ${n\choose 2}+1$.
\end{conjecture}

This conjecture is supported by computer results for $n<7$.
The computations also hint that the number of antipodes of
$e :=((-1,1), (1,2), (2,3), \ldots, (n-1,n))$ 
(i.e., the number of elements of maximal rank in the associated weak order) 
is given by sequence A054441 in~\cite{Sloane}
which involves central binomial coefficients, namely
the Catalan numbers of type $B$.

\bigskip

\noindent{\bf Acknowledgements.}
Thanks to C.\ Athanasiadis, O.\ Bernardi, A.\ Kalka, C.\ Keller, H.\ Y.\ Khachatryan,
S.\ Margolis, P.\ McNamara, A.\ Postnikov and V.\ Reiner
for useful discussions, comments and references.
Thanks also to the anonymous referees for their stimulating comments.

\end{document}